\documentclass[a4paper, 12pt]{article}
\makeindex
\usepackage{amscd}
\usepackage{amsmath, amssymb, amsfonts, amsthm, latexsym}
\usepackage[latin1]{inputenc}
\usepackage[all]{xy}

\DeclareMathOperator{\Hom}{Hom}

\DeclareMathOperator{\Image}{Im}

\DeclareMathOperator{\cross}{cr}
\DeclareMathOperator{\Sym}{Sym}

\DeclareMathOperator{\module}{\textrm{-}\bf{mod}}
\DeclareMathOperator{\Tot}{Tot}
\DeclareMathOperator{\Kos}{Kos}
\DeclareMathOperator{\Tor}{Tor}
\DeclareMathOperator{\hTor}{\bf{Tor}}

\newtheorem{thm}{Theorem}[section]
\newtheorem{Lemma}[thm]{Lemma}

\newtheorem{prop}[thm]{Proposition}

\theoremstyle{definition}
\newtheorem{defn}[thm]{Definition}

\theoremstyle{remark}
\newtheorem{remark}[thm]{Remark}

\begin{document}

\title{On the Derived Functors of the Third Symmetric-Power Functor}
\author{Bernhard K\"ock and Ramesh Satkurunath}
\date{\today}

\maketitle

\begin{quote}
{\footnotesize {\bf Abstract.} We compute the derived functors of
the third symmetric-power functor and their cross-effects for
certain values. These calculations match predictions by the first
named author and largely prove them in general.

}

\end{quote}

\section{Introduction}

Let \(R\) and \(S\) be rings. We recall that the construction of
the left derived functors \(L_k F:R\module\rightarrow S\module\)
of any covariant right-exact functor \(F:R\module\rightarrow
S\module\) is achieved by applying three functors. The first
functor constructs a projective resolution \(P.\) of the
\(R\)-module \(M\) that we wish to calculate the derived functor
of. Then the functor \(F\) is applied to the resolution \(P.\)
giving the chain complex \(F(P.)\).  Lastly the \(k^{th}\) derived
functor \(L_k F\) is defined to be \(H_k(F(P.))\), the \(k^{th}\)
homology of the chain complex \(F(P.)\). However for a given
module \(M\) the projective resolution of \(M\) is unique only up
to chain-homotopy equivalence, so this construction crucially
depends on the fact that \(F\) preserves chain-homotopies. In
general this fact does not hold when \(F\) is a nonlinear functor
such as the \(l^{th}\) symmetric-power functor, \(\Sym^{l}\), or
the \(l^{th}\) exterior-power functor, \(\Lambda^l\).  In the
paper \cite{DP} Dold and Puppe overcome this problem and define
the derived functors of non-linear functors by passing to the
category of simplicial complexes using the Dold-Kan
correspondence.

The Dold-Kan correspondence gives a pair of functors \(\Gamma\) and \(N\)
that provide an equivalence between the category of bounded chain
complexes and the category of simplicial complexes;
under this correspondence chain homotopies correspond to simplicial homotopies.
Furthermore in the simplicial world all functors preserve simplicial homotopy (not just linear functors).
Because of this the above definition of the derived functors of \(F\) becomes well defined
for any functor when \(F(P.)\) is replaced by the complex \(NF\Gamma(P.)\).

Now let $R$ be a Noetherian commutative ring,
let $I$ be an ideal in $R$
which is locally generated by a regular sequence of length $d$,
let $V$ be a finitely generated projective $R/I$-module and let
$P.(V)$ be an $R$-projective resolution of $V$.
In \cite{Ko1} the first named author explicitly calculates the modules
$H_k N\Sym^l\Gamma P.(V)$ when $d=1$ and $l$ is any positive integer,
and also when $d=2$ and $l=2$ (see Theorems 3.2 and 6.4 in \cite{Ko1}).
As explained in Section \ref{second cross section} such calculations lie at the heart of a new approach to the seminal
Adams-Riemann-Roch Theorem and hence to Grothendieck's Riemann-Roch theory.

In this paper we look at the case when $d=2$ and $l=3$.  Throughout this paper
we use the symbol $G_k$ to denote the derived functor
\[
\begin{array}{rll}
G_k:&\mathcal{P}_{R/I} &\rightarrow R\module \\
&V &\mapsto H_k N\Sym^3\Gamma P.(V)
\end{array}
\]
where \(\mathcal{P}_{R/I}\) denotes the category of finitely generated projective $R/I$-modules.
Let $L^3_1$ denote the Schur functor indexed with Young diagram of shape $(2,1)$
(see Definition \ref{Schur def}).
In Example 6.6 of \cite{Ko1} the first named author made the following prediction about
the functor~\(G_k\)
\[
G_k(V)\cong
\begin{cases}
\Sym^3(V) & k=0\\
L_1^3(V)\otimes I/I^2 & k=1\\
L^3_1(V)\otimes I/I^2 \otimes \Lambda^2(I/I^2) & k=3\\
D^3(V)\otimes \Lambda^2(I/I^2)^{\otimes 2} & k= 4\\
0 & k\ge5
\end{cases}
\]
and for the case when \(k=2\) he suggests that there exists an exact sequence:
\begin{align*}
0\rightarrow D^2(V)\otimes V \otimes \Lambda^2(I/I^2)
&\rightarrow H_2N\Sym^3\Gamma(P.(V))\\
&\rightarrow \Lambda^3(V)\otimes\Sym^2(I/I^2)
\rightarrow 0\text{.}
\end{align*}
Note that the prediction for $G_3(V)$ is given in slightly different terms
but, as will be explained after Proposition \ref{L^3_1},
this formulation is equivalent.

For each $k$ let $F_k$ stand for the prediction made for $G_k$ to be
(note that $F_2$ is defined only on objects, not on morphisms).
We will show that these predictions are true if $V$ is the free module $R/I$ of rank 1:
\[
G_k(R/I)\cong F_k(R/I),
\]
see Theorem \ref{first cross}.
Moreover we will show that the following isomorphisms hold for the higher cross-effects
(see Section \ref{cross-effect section} for the definition of cross-effect functors):
\begin{align*}
\cross_2(G_k)(R/I,R/I) &\cong \cross_2(F_k)(R/I,R/I) \\
\cross_3(G_k)(R/I,R/I,R/I) &\cong \cross_3(F_k)(R/I,R/I,R/I),
\end{align*}
see Proposition \ref{predictions} and Theorems \ref{second cross} and \ref{third cross}.

Theorem 1.5 of \cite{Ko1} implies that, if these isomorphisms
commute with certain structural maps, then the predictions are
true in general. To verify that the isomorphisms commute with the
maps as required, we would need to not only complete some quite
involved calculations, but we would also require a proper
definition of the functor $F_2$ i.e.\ one which applies to objects
and morphisms (rather than just to objects). Finding such a
suitable candidate for $F_2$ remains an open problem.

\section{Cross-effect Functors}\label{cross-effect section}

In \cite{EM} Eilenberg and Mac Lane introduced the theory of cross-effect functors.
The theory of cross-effect functors are central to the calculations of this paper.
In \cite{Ko1} the first named author proved a result (Theorem 1.5)
which shows that two functors are isomorphic if they agree
on certain data given by their cross-effect functors.
In this section we set the scene by introducing some of the theory of cross-effect functors from \cite{EM}.
Then we introduce the aforementioned theorem from \cite{Ko1} (see Theorem \ref{B's comparison theorem})
and exemplify it by showing that the Schur functor \(L^3_1\) and
the co-Schur functor \(\tilde{L}^3_1\) are isomorphic.
Furthermore we begin calculating the cross-effect functors of $G_k$ and calculate
$F_k(R/I), \cross_2(F_k)(R/I,R/I)$ and $\cross_3(F_k)(R/I,R/I,R/I).$

Let $\mathcal{P}$ be an additive category,
let $\mathcal{M}$ be an abelian category,
and let $F:\mathcal{P}\rightarrow\mathcal{M}$ be a functor with $F(0)=0$.

\begin{defn}
Let $k\ge0$. For any $V_1,\ldots,V_k\in\mathcal{P}$ and $1\le i\le k,$ let
$$P_i:V_1\oplus\ldots\oplus V_k\rightarrow V_i\rightarrow V_1\oplus\ldots\oplus V_k$$
denote the $i^{th}$ projection.  The $k$-functor
\[
\begin{array}{rrl}
\cross_k(F):&\mathcal{P}^k &\rightarrow \mathcal{M}\\
           &(V_1,\ldots,V_k) &\mapsto \Image\Big{(}\sum_{j=1}^k\sum_{1\le i_1\le \ldots \le i_j \le k}F(p_{i_1}+\ldots+p_{i_j})\Big{)}
\end{array}
\]
is called the \emph{$k^{\rm{th}}$ cross-effect of $F$};
here for any $V_1,\ldots,V_k,W_1,\ldots,W_k\in\mathcal{P}$
and $f_1\in\Hom_\mathcal{P}(V_1,W_1),\ldots,f_k\in\Hom_\mathcal{P}(V_k,W_k)$,
the map
$$
\cross_k(F)(f_1,\ldots,f_k):\cross_k(F)(V_1,\ldots,V_k)\rightarrow\cross_k(F)(W_1,\ldots,W_k)
$$
is induced by $f_1\oplus\ldots\oplus f_k\in
\Hom_\mathcal{P}(V_1\oplus\ldots\oplus V_k,W_1\oplus\ldots\oplus W_k).$
The functor $F$ is said to be \emph{of degree less than $k$} if
$\cross_k(F)$ is identically zero.
\end{defn}

The cross-effect functors $\cross_k(F), k\ge 0,$ have the following properties.
We obviously have $\cross_0(F)\equiv 0$ and $\cross_1(F)=F.$  Furthermore,
we have $\cross_k(F)(V_1,\ldots,V_k)=0$ if $V_i=0$ for any $i\in\{1,\ldots,k\}$
(see Theorem 9.2 on p.79 in \cite{EM}).
The most important property of cross-effects is given in the following proposition.
\begin{prop}\label{cross-effect decomp}
For any $k,l\ge1$ and $V_1,\ldots,V_l\in\mathcal{P},$ we have a canonical
isomorphism
\[
\cross_k(F)(\ldots,V_1\oplus\ldots\oplus V_l,\ldots)\cong
\bigoplus_{1\le j\le l} \bigoplus_{1\le i_1 < \ldots < i_j \le l}
\cross_{k+j-1}(\ldots,V_{i_1},\ldots,V_{i_j},\ldots)
\]
which is functorial in $V_1,\ldots,V_l.$ In particular,
$F$ is of degree $\le k,$ if and only if $\cross_k(F)$ is a $k$-additive functor.
\end{prop}

\begin{proof}
See Theorem 9.1, Lemma 9.8 and Lemma 9.9 in \cite{EM}.
\end{proof}

From Proposition \ref{cross-effect decomp} we see that
\[
F(V_1\oplus V_2)\cong F(V_1)\oplus F(V_2)\oplus\cross_2(F)(V_1,V_2)
\]
for all $V_1,V_2\in\mathcal{P}$; i.e.\ $\cross_2(F)$ measures the deviation
from linearity of the functor $F$.  This isomorphism can also be used to define
$\cross_2(F)$ (see section 3 in \cite{JM}.
Similarly, the higher cross-effects can be defined inductively by the isomorphism
\begin{align*}
\cross_k(F)(V_1,\ldots,V_{k-1},V_k\oplus V'_k) \cong\,\, &
\cross_k(F)(V_1,\ldots,V_{k-1},V_k)\\
&\oplus\cross_k(F)(V_1,\ldots,V_{k-1},V'_k)\\
&\oplus \cross_{k+1}(F)(V_1,\ldots,V_{k-1},V_k,V'_k)
\end{align*}
(see section 7 in \cite{JM}).
When actually calculating cross-effect functors
it is this definition that we will use.

\begin{defn}\label{crossplusdiag}
Let \(l\ge k\ge 1, V_1\), let \(\ldots,V_k\in\mathcal{P}\),
and let \(\epsilon=(\epsilon_1,\ldots,\epsilon_k)\in\{1,\ldots,l\}^k\) with
\(|\epsilon|:=\sum_{i=1}^k\epsilon_i=l\).
The composition
\begin{align*}
\Delta_\epsilon:\cross_k(F)(V_1,\ldots,V_k)
\xrightarrow{\cross_k(F)(\Delta,\ldots,\Delta)}
\cross_k&(F)(V_1^{\epsilon_1},\ldots,V_k^{\epsilon_k})\\
&\overset{\pi}\twoheadrightarrow
\cross_l(F)(V_1,\ldots,V_1,\ldots,V_k,\ldots,V_k)
\end{align*}
of the map \(\cross_k(F)(\Delta,\ldots,\Delta)\)
(induced by the diagonal maps \(\Delta:V_i\rightarrow V_i^{\epsilon_i}, i=1,\ldots,k\))
with the canonical projection \(\pi\)
(according to Proposition \ref{cross-effect decomp})
is called the
\emph{diagonal map associated with $\epsilon$}.
The analogous composition
\begin{align*}
+_{\epsilon}:\cross_l(F)(V_1,\ldots,V_1,\ldots,V_k,\ldots,V_k)
\hookrightarrow
\cross_k&(F)(V_1^{\epsilon_1},\ldots,V_k^{\epsilon_k})
\\
&\xrightarrow{{\cross_k(F)(+,\ldots,+)}}
\cross_k(F)(V_1,\ldots,V_k)
\end{align*}
is called \emph{plus map associated with $\epsilon$}.
\end{defn}

The maps \(\Delta_\epsilon\) and \(+_\epsilon\) form natural transformations
between the functors \(\cross_k(F)\) and
\(\cross_l(F)\circ(\Delta_{\epsilon_1},\ldots,\Delta_{\epsilon_k})\)
from \(\mathcal{P}^k\) to \(\mathcal{M}\).
One easily sees that the map \(\Delta_\epsilon\) can be decomposed
into a composition of maps \(\Delta_\delta\) with \(\delta\in\{1,2\}^j\)
such that \(|\delta|=j+1\) and \(j\in\{k,\ldots,l-1\}\).
The same holds for \(+_\epsilon\).

\begin{thm}\label{B's comparison theorem}
Let $A$ be a ring, $\mathcal{M}$ an abelian category, $d\in\mathbb{N_+}$,
and
\[
F,G:(\text{f.g. projective $A$-modules})\rightarrow\mathcal{M}
\]
be two functors of degree \(\le d\) with \(F(0)=0=G(0)\).
Suppose that there exist isomorphisms
\[
\alpha_i(A,\ldots,A):\cross_i(F)(A,\ldots,A)\tilde\longrightarrow\cross_i(G)(A,\ldots,A), \qquad i=1,\ldots,d,
\]
which are compatible with the action of $A$ in each component
and which make the following diagrams commute for \(i\in \{1,\ldots,d-1\}\)
and \(\epsilon\in\{1,2\}^i\) with \(|\epsilon|=i+1:\)
\[
\xymatrix
{
\cross_i(F)(A,\ldots,A)\ar[r]\ar[d]^{\Delta_\epsilon}   & \cross_i(G)(A,\ldots,A)\ar[d]^{\Delta_\epsilon}   \\
\cross_{i+1}(F)(A,\ldots,A)\ar[r]                                   & \cross_{i+1}(G)(A,\ldots,A)   \\
}
\]
\[
\xymatrix
{
\cross_{i+1}(F)(A,\ldots,A)\ar[r]\ar[d]^{+_\epsilon}    & \cross_{i+1}(G)(A,\ldots,A)\ar[d]^{+_\epsilon}    \\
\cross_{i}(F)(A,\ldots,A)\ar[r]                                         & \cross_{i}(G)(A,\ldots,A)\text{.} \\
}
\]
Then the two functors $F$ and $G$ are isomorphic.
\end{thm}

\begin{proof}
See Theorem 1.5 of \cite{Ko1}.
\end{proof}

\begin{defn}\label{Schur def}
Let $A$ be a commutative ring.
For any finitely generated $A$-module $V$ we define $L_1^3(V)$
by the following exact sequences and we call $L^3_1$ the Schur functor
indexed by the Young diagram of $(2,1)$.
\[
\xymatrix
{
\Lambda^3(V)\ar[r] & \Lambda^2(V)\otimes V\ar[rr]\ar[dr] && V\otimes\Sym^2(V)\ar[r] & \Sym^3(V)\\
&&L^3_1(V)\ar[ur]\ar[dr]\\
&0\ar[ur]&&0\\
}
\]
And we define the co-Schur functor $\tilde{L}^3_1$ by the following exact sequence.
\[
\xymatrix
{
D^3(V)\ar[r] & D^2(V)\otimes V\ar[rr]\ar[dr] && V\otimes\Lambda^2(V)\ar[r] & \Lambda^3(V)\\
&&\tilde L^3_1(V)\ar[ur]\ar[dr]\\
&0\ar[ur]&&0\\}
\]
Here, the horizontal sequences are the well-known exact Koszul
and co-Koszul complexes, see Definition \ref{KdefofKos}.
\end{defn}

We now give an example of how Theorem \ref{B's comparison theorem} can be used,
by proving that the Schur functor \(L^3_1\) and the co-Schur functor \(\tilde{L}^3_1\) are isomorphic.

\begin{prop}\label{L^3_1}
\[\cross_k(L^3_1)(A,\ldots,A)\cong\cross_k(\tilde L^3_1)(A,\ldots,A)\cong
\begin{cases}
0           & k=1\\
A\oplus A   & k=2\\
A\oplus A   & k=3\\
0           & k\ge4\text{.}\\
\end{cases}\]
Moreover the two functors $L^3_1$ and $\tilde L^3_1$ are isomorphic.
\end{prop}

Note this justifies the reformulation of the predictions for $G_k$ as seen in the introduction,
as the original prediction for $G_3$ in \cite{Ko1} was written as
$G_3(V)\cong\tilde L_1^3(V) \otimes I/I^2 \otimes \Lambda^2(I/I^2)$.

\begin{proof}
Let $V$ and $W$ be finitely generated projective $A$-modules.
Definition \ref{Schur def} tells us the co-Schur functor $\tilde L_1^3(V)$
is given by the short exact sequence
 \(0\rightarrow \tilde L^3_1(V)\rightarrow V\otimes\Lambda^2(V)\rightarrow \Lambda^3(V)\rightarrow 0\).
In particular this tells us that \(\cross_1(\tilde L_1^3)(A)=0\), similarly
we see that \(L_1^3(A)=0\).

Next we want to compute \(\cross_2(\tilde L_1^3)(A,A)\) and \(\cross_2(L_1^3)(A,A)\).
We have the short exact sequence
\(0\rightarrow L^3_1(V\oplus W)\rightarrow (V\oplus W)\otimes\Lambda^2(V\oplus W)
\rightarrow \Lambda^3(V\oplus W)\rightarrow 0\).
Using the canonical decomposition for \(\Lambda^n\)
we get the following short exact sequence:
\begin{align*}
0\rightarrow &\cross_2(\tilde L^3_1)(V, W)
\rightarrow\\
&V\otimes(V\otimes W)\oplus V\otimes \Lambda^2(W)
\oplus W\otimes\Lambda^2(V)\oplus W\otimes(V\otimes W)\\
&\rightarrow \Lambda^2(V)\otimes W\oplus V\otimes\Lambda^2(W)\rightarrow 0\text{.}
\end{align*}
Therefore \(\cross_2(\tilde L^3_1)(V, W)\cong V^{\otimes 2}\otimes W\oplus  V\otimes W^{\otimes 2}\).
Similarly we find that \(\cross_2(L^3_1)(V, W)\cong V^{\otimes 2}\otimes W\oplus  V\otimes W^{\otimes 2}\),
so \(\cross_2(\tilde L^3_1)\) and \(\cross_2(L^3_1)\) are isomorphic as bi-functors.
Because of the way that higher cross-effects are calculated from lower cross-effects we see that,
for $k\ge 2$, the $k$-functors \(\cross_k(\tilde L^3_1)\) and \(\cross_k(L^3_1)\) will be isomorphic
and the corresponding diagrams of $+$ and \(\Delta\) maps will commute.
We easily verify that \(\cross_2(L^3_1)(A,A)\cong\cross_2(\tilde L^3_1)(A,A)\)
and \(\cross_3(L^3_1)(A,A,A)\cong\cross_3(\tilde L^3_1)(A,A,A)\) are free of rank 2 over $A$, as stated.

Since \(\cross_1(\tilde L^3_1)(A)\cong\cross_1(L^3_1)(A)\cong 0\)
the corresponding + and \(\Delta\) maps involving
\(\cross_1(\tilde L^3_1)(A)\) and \(\cross_1(L^3_1)(A)\) will just
be zero maps, and hence will commute with the isomorphism. So
Theorem \ref{B's comparison theorem} implies that \(\tilde L^3_1\)
is isomorphic to \(L^3_1\).
\end{proof}

The following lemma begins the work calculating the cross-effects of $G_k$
that we will complete in the following chapters.
Let $A$ and $A'$ be simplicial $R$-modules;
the reader should be aware that in the following lemma and its proof
that $A\otimes A'$ means the simplicial module whose $n^{\rm{th}}$
place is $A_n\otimes A'_n$.

\begin{Lemma}\label{firststep}
\begin{align*}
\cross_2(G_k)(V,W)\cong\,\,&
H_kN(\Sym^2\Gamma P.(V)\otimes \Gamma P.(W))\\
&\oplus H_kN(\Gamma P.(V)\otimes \Sym^2\Gamma P.(W))
\end{align*}
\[
\cross_3(G_k)(V,W,X)\cong
H_kN\big{(}\Gamma P.(V)\otimes \Gamma P.(W) \otimes \Gamma P.(X)\big{)}\text{.}
\]
\end{Lemma}

\begin{proof}
First we calculate \(G_k(V\oplus W)\) for $R$-modules $V,W$
to give us an expression for \(\cross_2(G_k)(V,W)\).
To do this we use the fact that \(P.,\Gamma, N\) and \(H_k\) are linear functors and also
the canonical decomposition
\(\Sym^3(V\oplus W)\cong \bigoplus_{l=0}^3\Sym^{3-l}(V)\otimes\Sym^l(W)\).
\begin{align*}
G_k(V\oplus W)=\,\,&H_kN\Sym^3\Gamma P.(V\oplus W)\\
\cong\,\, & H_k N\Big{(}\Sym^3\Gamma P.(V) \oplus \Sym^2\Gamma P.(V)\otimes \Gamma P.(W)\\
&\oplus  \Gamma P.(V)\otimes \Sym^2\Gamma P.(W) \oplus \Sym^3\Gamma P.(W)\Big{)}\\
\cong\,\, & G_k(V) \oplus H_kN\Big{(}\Sym^2\Gamma P.(V)\otimes \Gamma P.(W)\Big{)}\\
&\oplus H_kN\Big{(}\Gamma P.(V)\otimes \Sym^2\Gamma P.(W)\Big{)} \oplus  G_k(W)\text{.}
\end{align*}
And hence we get the desired expression for $\cross_2(G_k)(V,W)$.

We now complete the proof by using our expression for $\cross_2(G_k)(V,W),$
to get an expression for $\cross_3(G_k)(V,W,X)$
We start this by calculating \(\cross_2(G_k)(V, W \oplus X)\) for $R$-modules $V,W,X$.
To simplify our calculation we split up our expression for
\(\cross_2(G_k)(V,W\oplus X)\)
into
\(H_k N\big{(}\Sym^2\Gamma P.(V)\otimes \Gamma P.(W\oplus X)\big{)}\)
and
\(H_kN\big{(}\Gamma P.(V)\otimes \Sym^2\Gamma P.(W\oplus X)\big{)}\),
calculate each part separately, then add them together afterwards.
\begin{align*}
H_k N\Big{(}\Sym^2\Gamma P.(V)&\otimes \Gamma P.(W\oplus X)\Big{)}\\
\cong\,\,&
H_k N\Big{(}\Sym^2\Gamma P.(V)\otimes \Gamma P.(W)  \Big{)}\\
&\oplus H_k N\Big{(}\Sym^2\Gamma P.(V)\otimes \Gamma P.(X) \Big{)}\text{.}
\end{align*}
\begin{align*}
H_kN\Big{(}\Gamma P.(V)\otimes &\Sym^2\Gamma P.(W\oplus X)\Big{)}\\
&\cong
H_kN\Big{(}\Gamma P.(V)\otimes \Sym^2\Gamma P.(W))\Big{)}\\
&\oplus H_kN\Big{(}\Gamma P.(V)\otimes \big{(}\Gamma P.(W) \otimes \Gamma P.(X)\big{)}\Big{)}\\
&\oplus H_kN\Big{(}\Gamma P.(V)\otimes \Sym^2\Gamma P.(X)\Big{)}\text{.}
\end{align*}
This gives us the following expression for \(\cross_2(G_k)(V,W\oplus X)\):
\begin{align*}
\cross_2(G_k)(V,W\oplus X)\cong\,\,&
\cross_2(G_k)(V,W)\oplus\cross_2(G_k)(V,X)\\
&\oplus H_kN\Big{(}\Gamma P.(V)\otimes
\big{(}\Gamma P.(W) \otimes \Gamma P.(X)\big{)}\Big{)}\text{.}
\end{align*}
Therefore \(
\cross_3(G_k)(V,W,X)\cong
H_kN\big{(}\Gamma P.(V)\otimes \Gamma P.(W) \otimes \Gamma P.(X)\big{)}
\) as desired.
\end{proof}

We finally calculate the cross-effects of the predictions evaluated at the free $R/I$-module of rank 1.
\begin{prop}\label{predictions}
We have the following $R/I$-module isomorphisms:
\begin{align*}
F_k(R/I)
&\cong
\begin{cases}
R/I                                                 &   k=0\\
0                                                   &       k=1\\
\Lambda^2(I/I^2)                        &       k=2\\
0                                                       &       k=3\\
\Lambda^2(I/I^2)^{\otimes 2}&       k=4\\
\end{cases}
\\
\\
\cross_2(F_k)&(R/I,R/I) \\
&\cong
\begin{cases}
R/I \oplus R/I                                                          &   k=0\\
I/I^2 \oplus I/I^2                                                      &   k=1\\
\Lambda^2(I/I^2)\oplus\Lambda^2(I/I^2)\oplus\Lambda^2(I/I^2)\oplus\Lambda^2(I/I^2)  &   k=2\\
I/I^2 \otimes \Lambda^2(I/I^2) \oplus I/I^2 \otimes \Lambda^2(I/I^2)    &   k=3\\
\Lambda^2(I/I^2)^{\otimes 2} \oplus \Lambda^2(I/I^2)^{\otimes 2}        &   k=4\\
\end{cases}
\\
\\
\cross_3(F_k)&(R/I,R/I,R/I)\\
&\cong
\begin{cases}
R/I                                                                 &   k=0\\
I/I^2 \oplus I/I^2                                                  &   k=1\\
\Lambda^2(I/I^2) \oplus (I/I^2)^{\otimes 2} \oplus \Lambda^2(I/I^2) \oplus \Sym^2(I/I^2)                                        &   k=2\\
I/I^2 \otimes \Lambda^2(I/I^2) \oplus I/I^2 \otimes \Lambda^2(I/I^2)    &   k=3\\
\Lambda^2(I/I^2)^{\otimes 2}                                        &   k=4\text{.}\\
\end{cases}
\end{align*}
\end{prop}

\begin{proof}
These results follow from a simple application of Proposition
\ref{L^3_1}, and the canonical decompositions for the symmetric-,
exterior-, and divided-power functors.
\end{proof}

\section{The Iterated Eilenberg-Zilber Theorem and Calculating \(\cross_3(G_k)(R/I,R/I,R/I)\)}

Let $A^1$ and $A^2$ be simplicial $R$-modules.
The Eilenberg-Zilber Theorem
(see \S 28 of \cite{May},
specifically Corollary 29.6 on p132)
tells us
that \(N\Delta(A^1\otimes A^2)\) is chain homotopic to \(\Tot(NA^1\otimes NA^2)\)
(where $\Delta(A^1\otimes A^2)$ denotes the diagonal of the bisimplicial complex
$A^1\otimes A^2$).
Let $C^1, C'^1, C^2, C'^2$ be chain complexes of $R$-modules.
If $C^1$ is chain homotopic to $C'^1$ and $C^2$ is chain homotopic to $C'^2$
then $\Tot(C^1\otimes C^2)\cong\Tot(C'^1\otimes C'^2)$; because of this the
Eilenberg-Zilber Theorem can be iterated to give us the following theorem.
\begin{thm} \emph{\bf{The Iterated Eilberg-Zilber Theorem.}}\label{iEZ}
Let \(n\in\mathbb{N}\) with \(n\ge 2\) and let \(A^1,\ldots, A^n\) be simplicial complexes.
Then the complexes \(N \Delta (A^1 \otimes \ldots \otimes A^n)\) and \(\Tot(NA^1\otimes \ldots \otimes NA^n)\)
are chain homotopic and (consequently) they are quasi-isomorphic:
\[H_k(N \Delta (A^1 \otimes \ldots \otimes A^n)) \cong H_k(\Tot(NA^1\otimes \ldots \otimes NA^n)).
\]
\end{thm}

The following theorem, together with the canonical decomposition
for exterior powers, shows that
the third cross-effect functor for
the derived functors of $\Sym^3$ evaluated on $(R/I,R/I,R/I)$
matches the predictions (see Proposition \ref{predictions}).
\begin{thm}\label{third cross}
\[\cross_3(G_k)(R/I,R/I,R/I)\cong\Lambda^k(I/I^2\oplus I/I^2)\]
\end{thm}

\begin{proof}
From Lemma \ref{firststep} we know that
\[
\cross_3(G_k)(V,W,X)\cong
H_kN\big{(}\Gamma P.(V)\otimes \Gamma P.(W) \otimes \Gamma P.(X)\big{)}\text{.}
\]
Here \(\Gamma P.(V)\otimes \Gamma P.(W) \otimes \Gamma P.(X)\)
stands for the simplicial complex whose $k^{\rm{th}}$ place is
\(\Gamma P_k(V)\otimes \Gamma P_k(W) \otimes \Gamma P_k(X).\)
But we may consider this to be the diagonal of the tri-simplicial complex
whose $(k,l,m)^{\rm{th}}$ place is
\(\Gamma P_k(V)\otimes \Gamma P_l(W) \otimes \Gamma P_m(X),\)
and by doing so we write
\[
\cross_3(G_k)(V,W,X)\cong
H_kN\Delta\big{(}\Gamma P.(V)\otimes \Gamma P.(W) \otimes \Gamma P.(X)\big{)}\text{.}
\]

The Iterated Eilenberg-Zilber Theorem tells us that
\[
H_kN\Delta\big{(}\Gamma P.(R/I)\otimes \Gamma P.(R/I) \otimes \Gamma P.(R/I)\big{)}
\cong H_k\Tot\big{(}P.(R/I)\otimes P.(R/I)\otimes P.(R/I)\big{)}\text{.}
\]
Theorem 5.1 of \cite{Ko1} tells us that
\[
H_k\Tot\big{(}P.(R/I)\otimes P.(R/I)\otimes P.(R/I)\big{)} \cong \Lambda^k(I/I^2\oplus I/I^2)\text{,}
\]
as desired.
\end{proof}

\section{The Derived Functors of $\Sym^2$ and Calculating \(\cross_2(G_k)(R/I,R/I)\)}\label{second cross section}

The following theorem shows that
the second cross-effect functor of
the derived functors of $\Sym^3$ evaluated on $(R/I,R/I)$
matches the predictions (see Proposition \ref{predictions}).
Essential to the following proof is the information that has already been calculated
about $H_kN\Sym^2\Gamma P.$ in \cite{Ko1}, and the use of the Hypertor functor to exploit this.
\begin{thm}\label{second cross}
\[\cross_2(G_k)(R/I,R/I)
\cong
\begin{cases}
R/I\oplus R/I                                                                       & k=0\\
I/I^2\oplus I/I^2                                                                   & k=1\\
\Lambda^2(I/I^2)\oplus\Lambda^2(I/I^2)\oplus \Lambda^2(I/I^2)\oplus\Lambda^2(I/I^2) & k=2\\
I/I^2\otimes \Lambda^2(I/I^2)\oplus I/I^2\otimes \Lambda^2(I/I^2)                   & k=3\\
\Lambda^2(I/I^2)^{\otimes 2}\oplus \Lambda^2(I/I^2)^{\otimes 2}                     & k=4\\
0                                                                                   & k\ge5\text{.}\\
\end{cases}\]
\end{thm}
\begin{proof}
Lemma \ref{firststep} tells us that
\begin{align*}
\cross_2(G_k)(V,W)\cong\,\,&
H_kN(\Sym^2\Gamma P.(V)\otimes \Gamma P.(W))\\
&\oplus H_kN(\Gamma P.(V)\otimes \Sym^2\Gamma P.(W))\text{.}
\end{align*}
Here $\Sym^2\Gamma P.(V)\otimes \Gamma P.(W)$ and $\Gamma P.(V)\otimes \Sym^2\Gamma P.(W)$
stand for simplicial modules whose $k^{\rm{th}}$ place is
$\Sym^2\Gamma P_k(V)\otimes \Gamma P_k(W)$ and $\Gamma P_k(V)\otimes \Sym^2\Gamma P_k(W)$
respectively;
however we may consider them to be the diagonal of the bi-simplicial modules whose $(k,l)^{\rm{th}}$ place is
$\Sym^2\Gamma P_k(V)\otimes \Gamma P_l(W)$ and $\Gamma P_k(V)\otimes \Sym^2\Gamma P_l(W)$ respectively,
therefore we can write
\begin{align*}
\cross_2(G_k)(V,W)\cong\,\,&
H_kN\Delta(\Sym^2\Gamma P.(V)\otimes \Gamma P.(W))\\
&\oplus H_kN\Delta(\Gamma P.(V)\otimes \Sym^2\Gamma P.(W))\text{.}
\end{align*}

We now start to calculate $H_kN\Delta(\Sym^2\Gamma P.(V)\otimes \Gamma P.(W))$,
later we will calculate $H_kN\Delta(\Gamma P.(V)\otimes \Sym^2\Gamma P.(W))$
and then add the two together to get an expression for $\cross_2(G_k)(V,W)$.

Using the Eilenberg-Zilber Theorem (Theorem \ref{iEZ}) we see that
\begin{align*}
H_kN\Delta(\Sym^2(\Gamma P.(V)\otimes\Gamma P.(W))
\cong\,\, & H_k\Tot(N\Sym^2\Gamma P.(V)\otimes N\Gamma P.(W))\\
\cong\,\, & H_k\Tot(N\Sym^2\Gamma P.(V)\otimes P.(W))\text{.}
\end{align*}
So we want to calculate \(H_k\Tot(N\Sym^2\Gamma P.(V)\otimes P.(W))\),
but this is just the definition of the hypertor
\(\hTor^R_i(N\Sym^2\Gamma P.(V),W)\).
Application 5.7.8 of \cite{Weibel}
gives us a spectral sequence to calculate hypertor
\[
^{II}E^2_{pq}=\Tor_p(H_q(A),B)\Rightarrow\hTor^R_{p+q}(A_*,B)\text.{}
\]
Theorem 6.4 of \cite{Ko1} tells us that
\[
H_kN\Sym^2\Gamma(P.(V))\cong
\begin{cases}
\Sym^2(V)                       & k=0\\
\Lambda^2(V)\otimes I/I^2           & k=1\\
D^2(V)\otimes\Lambda^2(I/I^2)   & k=2\\
0                               & k\ge3\text{.}\\
\end{cases}
\]
Now \(\Sym^2(R/I)\cong R/I, \Lambda^2(R/I)\cong 0\)  and \(D^2(R/I)\cong R/I\).
Hence the terms in the $2^{\rm{nd}}$ sheet of our spectral sequence are
$\Tor_p(R/I,R/I)$ in the zeroth column,
$\Tor_p(0,R/I)=0$ in the first column,
$\Tor_p(\Lambda(I/I^2),R/I)$ in the second column and $0$ everywhere else.
From from Example 5.2 of \cite{Ko1} (also see Theorem 5.1 of \cite{Ko1})
we know that \(\Tor_k(R/I,R/I)\cong\Lambda^k(I/I^2)\)
and hence have \(\Tor_k(V,W)\cong V\otimes W\otimes\Lambda^k(I/I^2)\).
So the $2^{\rm{nd}}$ sheet of
the spectral sequence looks like this:
\[
\xymatrix
{
\ddots& \vdots&\vdots&\vdots&\vdots\\
\ldots& 0&0&0&0\\
\ldots& 0&\Lambda^2(I/I^2)^{\otimes 2}  &0&\Lambda^2(I/I^2)\\
\ldots& 0&I/I^2\otimes\Lambda^2(I/I^2)  &0&I/I^2\\
\ldots& 0&\Lambda^2(I/I^2)              &0&R/I\text{.}\\
}
\]
The differentials on this level of the spectral sequence are
$-2$ in the $p$-direction and $+1$ in the $q$-direction,
so each differential either comes from or goes to a zero module.
Hence the differentials are all zero maps
i.e.\ the spectral sequence has already converged on the second level.

Therefore
\[
H_k N\Delta\big{(}\Sym^2\Gamma P.(R/I)\otimes \Gamma P.(R/I)\big{)}
\cong
\begin{cases}
R/I                                         & k=0\\
I/I^2                                       & k=1\\
\Lambda^2(I/I^2)\oplus\Lambda^2(I/I^2)      & k=2\\
I/I^2\otimes \Lambda^2(I/I^2)               & k=3\\
\Lambda^2(I/I^2)^{\otimes 2}                & k=4\\
0                                           & k\ge5\text{.}\\
\end{cases}\]
By symmetry
\[
H_kN\Delta\big{(}\Gamma P.(R/I)\otimes \Sym^2\Gamma P.(R/I)\big{)}
\cong
H_k N\Delta\big{(}\Sym^2\Gamma P.(R/I)\otimes \Gamma P.(R/I)\big{)}\text{.}
\]
Hence \[\cross_2(G_k)(R/I,R/I)
\cong
\begin{cases}
R/I\oplus R/I                                                                       & k=0\\
I/I^2\oplus I/I^2                                                                   & k=1\\
\Lambda^2(I/I^2)\oplus\Lambda^2(I/I^2)\oplus \Lambda^2(I/I^2)\oplus\Lambda^2(I/I^2) & k=2\\
I/I^2\otimes \Lambda^2(I/I^2)\oplus I/I^2\otimes \Lambda^2(I/I^2)                   & k=3\\
\Lambda^2(I/I^2)^{\otimes 2}\oplus \Lambda^2(I/I^2)^{\otimes 2}                     & k=4\\
0                                                                                   & k\ge5\text{.}\\
\end{cases}\]
\end{proof}

\section{The Cauchy Decomposition of $\Sym^3(P\otimes Q)$ and Calculating \(G_k(R/I)\)}

In this final section we complete our calculations of $G_k$,
but before we proceed we introduce
the Cauchy decomposition given in chapter III of \cite{ABW}
as it applies to the third symmetric power,
and remind the reader of Koszul and co-Koszul complexes.

Let $R$ be a ring. Let $P$ and $Q$ be finitely generated
projective $R$-modules. We now summarize the Cauchy decomposition
of $\Sym^3(P\otimes Q).$ This decomposition plays a central role
in Theorem \ref{first cross}, our calculation of \(G_k(R/I)\).

A three step filtration
\begin{align*}
0\subset
M_{(3)}(\Sym^3(P\otimes Q))\subset
M_{(2,1)}(&\Sym^3(P\otimes Q))
\subset
\Sym^3(P\otimes Q)
\end{align*}
is put on \(\Sym^3(P\otimes Q)\).
The \(M_{(3)}(\Sym^3(P\otimes Q))\) part is defined to be the image of the determinant
map
\begin{align*}
\Lambda^3 P\otimes \Lambda^3 Q \rightarrow& \Sym^3(P\otimes Q) \\
p_1\wedge p_2\wedge p_3 \otimes q_1\wedge q_2\wedge q_3 \mapsto&
\left| \begin{array}{ccc}
p_1\otimes q_1 & p_1\otimes q_2 & p_1\otimes q_3 \\
p_2\otimes q_1 & p_2\otimes q_2 & p_2\otimes q_3 \\
p_3\otimes q_1 & p_3\otimes q_2 & p_3\otimes q_3 \end{array} \right|
\end{align*}
(note this is simply isomorphic to \(\Lambda^3 P\otimes \Lambda^3 Q\)).
The \(M_{(2,1)}(\Sym^3(P\otimes Q))\) part is defined to be equal to the previous part
\(M_{(3)}(\Sym^3(P\otimes Q))\) part plus the image of the following homomorphism:
\begin{align*}
\Lambda^2 P\otimes P\otimes \Lambda^2 Q \otimes Q \rightarrow& \Sym^3(P\otimes Q) \\
p_1\wedge p_2\otimes p_3 \otimes q_1\wedge q_2\otimes q_3 \mapsto&
\left| \begin{array}{cc}
p_1\otimes q_1 & p_1\otimes q_2 \\
p_2\otimes q_1 & p_2\otimes q_2  \\\end{array} \right|
(p_3\otimes q_3)\text{.}
\end{align*}

The quotients of this filtration are given by the short exact sequences
\[
0\rightarrow \Lambda^3 P\otimes \Lambda^3 Q \rightarrow M_{(2,1)}(\Sym^3(P\otimes Q))
\rightarrow L^3_1P\otimes L^3_1Q \rightarrow 0
\]
and
\[
0\rightarrow M_{(2,1)}(\Sym^3(P\otimes Q)) \rightarrow \Sym^3(P\otimes Q)
\rightarrow
\Sym^3P\otimes\Sym^3Q \rightarrow 0\text{.}
\]

We now properly introduce the definition of Koszul complexes,
which we use for projective resolutions.

\begin{defn}\label{KdefofKos}
Let \(f:P\rightarrow Q\) be a homomorphism between two finitely generated projective
$R$-modules, and \(n\in\mathbb{N}\).
Let $\Kos^n(f)$ be the \emph{Koszul complex}
\[
0\rightarrow\Lambda^nP\overset{d_{n-1}}\rightarrow\Lambda^{n-1}P\otimes Q\overset{d_{n-2}}\rightarrow \ldots\overset{d_1}\rightarrow P\otimes\Sym^{n-1}Q\overset{d_0}\rightarrow \Sym^nQ\rightarrow 0
\]
where, for \(k\in\{0,1,\ldots,n-1\}\), the differential
\[d_k:\Lambda^{k+1}P\otimes\Sym^{n-k-1}Q\rightarrow
\Lambda^{k}P\otimes\Sym^{n-k}Q\]
acts by
\[
p_1\wedge\ldots\wedge p_{k+1}\otimes q_{k+2}...q_n\mapsto
\Sigma_{i=1}^{k+1}(-1)^{k+1-i}p_1\wedge\ldots\wedge \hat{p}_i\wedge\ldots\wedge p_{k+1}
\otimes f(p_i)q_{k+2}...q_n\text{.}
\]

Now let $f^*:Q^*\rightarrow P^*$ denote the dual map,
then the part of the Koszul complex \(\Kos^n(f^*)\) in the $k^{\rm{th}}$ degree
is \(\Lambda^kQ^*\otimes\Sym^{n-k}P^*\).
The dual of this chain complex is a co-chain complex with
the part in the $k^{\rm{th}}$ degree being
\(\big{(}\Lambda^kQ^*\otimes\Sym^{n-k}P^*\big{)}^*
\cong \Lambda^kQ\otimes D^{n-k}P\), i.e.\
\[
0\leftarrow\Lambda^nQ\leftarrow\Lambda^{n-1}Q\otimes P\leftarrow \ldots\leftarrow Q\otimes D^{n-1}P\leftarrow D^nP\leftarrow 0\text{.}
\]
We call this the \emph{co-Koszul complex} and denote it by
\(\tilde\Kos^n(f)\).
\end{defn}

\begin{remark}
It is well known that
the complexes \(\Kos(f)\) and \(\tilde\Kos(f)\) are exact if $f$ is an isomorphism.
\end{remark}

The following two propositions will be useful in the proof of
Theorem \ref{first cross}.
\begin{prop}\label{Kqi}
Let \(f:P\rightarrow Q\) be a homomorphism between two finitely
generated projective $R$-modules.  If we consider \(P\rightarrow
Q\) to be a chain complex concentrated in degrees $1$ and $0$ then
we have quasi-isomorphisms
\[\Kos^n(f)\cong N\Sym^n\Gamma(P\rightarrow Q)\]
and
\[\tilde{\Kos}^n(f)\cong N\Lambda^n\Gamma(P\rightarrow Q)\text{,}\]
where \(\Gamma\) and $N$ are the functors of the Dold-Kan correspondence.
\end{prop}

\begin{proof}
See Proposition 2.4 and Remark 3.6 of \cite{Ko1}.
\end{proof}

\begin{prop}\label{hTor}
Let $A.$ and $B.$ be bounded complexes of finitely generated projective $R$-modules.
Then we have a spectral sequence
\[
^{II}E^2_{p,q}=\bigoplus_{q=q'+q''}{\rm
Tor}_p(H_{q'}(A.),H_{q''}(B.)) \Rightarrow \hTor_{p+q}(A.,B.)
\]
where
$\hTor_n(A.,B.)$ is defined as $H_n\Tot(A.\otimes B.)$.  In particular
if $A.'$ and $B.'$ are further complexes as above and
\(A\rightarrow A'\) and \(B\rightarrow B'\) are quasi-isomorphisms then the induced morphism
\[
\Tot(A.\otimes B.)\rightarrow \Tot(A.'\otimes B.')
\]
is a quasi-isomorphism as well.
\end{prop}

\begin{proof}
See Application 5.7.8 of \cite{Weibel}.
\end{proof}

The following theorem shows that if $I$ is globally generated
by a regular sequence
then the derived functors of $\Sym^3$ evaluated on $R/I$
match the predictions.
\begin{thm}\label{first cross}
If $I$ is generated by a regular sequence of length $2$ then
the module \(G_k(R/I)\) is a free $R/I$-module of rank $1$,
for $k=0,2$ or $4$ and otherwise of rank $0$.
\end{thm}

\begin{proof}
Let $f,g$ be a regular sequence in $R$,
and let $I$ be generated by it.
Also we let $K.$ denote the complex
$\ldots\rightarrow 0 \rightarrow R \overset{f}\rightarrow R$ and
$L.$ denote the Koszul complex $\ldots\rightarrow 0 \rightarrow R \overset{g}\rightarrow R$.
We use the complex
$$\Kos^2( R \oplus R \overset{(f,g)}\rightarrow R)
\cong\Tot(K.\otimes L.)$$
as a resolution of $R/I$.

We see that:
\begin{align*}
G_k(R/I):=&H_kN\Sym^3\Gamma\Tot(K. \otimes L.) \\
\cong\,\,& H_k N \Sym^3 \Gamma \Tot(N\Gamma K.\otimes N\Gamma
L.)\text{.}
\end{align*}

Theorem \ref{iEZ} tells us that
\(\Tot(N\Gamma K.\otimes N\Gamma L.)\) is chain homotopic to \(N\Delta(\Gamma K. \otimes \Gamma L.)\).
Applying \(\Gamma\) turns the notion of chain homotopy into simplicial homotopy,
all functors preserve homotopy in the simplicial world
and \(N\) changes the notion of simplicial homotopy into the notion of chain homotopy.
So \(N\Sym^3\Gamma\) turns the chain homotopy between
\(\Tot(N\Gamma K.\otimes N\Gamma L.)\) and \(N\Delta(\Gamma K. \otimes \Gamma L.)\)
into a chain homotopy between \(N\Sym^3\Gamma\Tot(N\Gamma K.\otimes N\Gamma L.)\) and
\(N\Sym^3\Gamma N\Delta(\Gamma K. \otimes \Gamma L.)\).
Chain homotopic complexes are quasi-isomorphic,
so continuing our calculation of \(G_k(R/I)\) where we left off we get:
\begin{align*}
G_k(R/I)&\cong H_k N\Sym^3\Gamma N\Delta( \Gamma K.\otimes \Gamma L.)
\cong H_k N\Sym^3\Delta(\Gamma K.\otimes \Gamma L.)\\
&\cong H_k N\Delta\Sym^3(\Gamma K.\otimes \Gamma L.)\text{.}
\end{align*}

Now we calculate \(G_k(R/I)\) by calculating \(H_k N\Delta\Sym^3(\Gamma K.\otimes \Gamma L.)\).
We cannot calculate this directly, so instead we employ
the short exact sequences that come from the Cauchy decomposition;
these short exact sequences
will allow us to get information about the homologies of
\(N\Delta\Sym^3(\Gamma K.\otimes \Gamma L.)\) from easier to calculate homologies.

From the Cauchy decomposition we get the following short exact sequences of bisimplicial modules:
\[
0\rightarrow \Lambda^3 \Gamma K.\otimes \Lambda^3 \Gamma L. \rightarrow M_{(2,1)}(\Sym^3(\Gamma K.\otimes \Gamma L.))
\rightarrow L^3_1 \Gamma K.\otimes L^3_1 \Gamma L. \rightarrow 0
\]
\begin{align*}
0
\rightarrow
M_{(2,1)}(\Sym^3(\Gamma K.\otimes \Gamma L.))
&\rightarrow
\Sym^3(\Gamma K. \otimes \Gamma L.)\\
&\rightarrow
\Sym^3\Gamma K.\otimes\Sym^3\Gamma L.
\rightarrow 0\text{.}
\end{align*}
Applying \(N\Delta\) to this gives us a short exact sequence of chain complexes.
We can turn the homologies of these into two long exact sequences,
this will allow us to get information about the homologies of \(M_{(2,1)}(\Sym^3(\Gamma K.\otimes \Gamma L.))\)
from the easier to calculate homologies of \(\Lambda^3 \Gamma K.\otimes \Lambda^3 \Gamma L.\)
and \(L^3_1\Gamma K.\otimes L^3_1\Gamma L.\).
This information about the homologies of \(M_{(2,1)}(\Sym^3(\Gamma K.\otimes \Gamma L.))\)
together with the homologies of the easier to calculate homologies of \(\Sym^3\Gamma K.\otimes\Sym^3\Gamma L.\)
will tell us the ranks of the homologies of \(\Sym^3(\Gamma K.\otimes \Gamma L.)\).

First we calculate the homologies of \(L^3_1\Gamma K.\)
The definition of \(L^3_1\) gives us the following short exact sequence
for any finitely generated projective module $P$
\[
0\rightarrow L^3_1P \rightarrow P\otimes \Sym^2 P \rightarrow \Sym^3 P \rightarrow 0\text{,}
\]
which gives us the short exact sequence of simplicial complexes
\[
0\rightarrow L^3_1\Gamma K. \rightarrow \Gamma K. \otimes \Sym^2 \Gamma K. \rightarrow \Sym^3 \Gamma K. \rightarrow 0\text{,}
\]
the middle term of this short exact sequence is the
simplicial complex whose $k^{\rm{th}}$ term is \(\Gamma K_k \otimes \Sym^2 \Gamma K_k\).
We think of this middle term instead as the diagonal of a bisimplicial complex
whose \((k,l)^{\rm{th}}\) term is \(\Gamma K_k \otimes \Sym^2 \Gamma K_l\).
Now applying the functor $N$ turns this into a short exact sequence of chain complexes
\[
0\rightarrow NL^3_1\Gamma K. \rightarrow N\Delta(\Gamma K. \otimes \Sym^2 \Gamma K.) \rightarrow N\Sym^3 \Gamma K. \rightarrow 0\text{,}
\]
and from this we can create a long exact sequence that gives us information about
\(H_kN L^3_1 \Gamma K.\).

Applying the Eilenberg-Zilber Theorem, then Propositions \ref{Kqi} and \ref{hTor}
we see that:
\begin{align*}
H_kN\Delta(\Gamma K.\otimes \Sym^2\Gamma K.)\cong\,\,& H_k\Tot(N\Gamma K.\otimes N\Sym^2\Gamma K.)\\
&\cong H_k\Tot(K.\otimes \Kos^2(f))\text{.}
\end{align*}
Now
\begin{align*}
\Kos^2(f)
&=(\Lambda^2(R)\rightarrow
\Lambda^1(R)\otimes\Sym^1(R)
\rightarrow\Sym^2(R))\\
&=(0\rightarrow R\overset{f}\rightarrow R)=K.\text{,}
\end{align*}
and therefore
\begin{align*}
H_kN\Delta(\Gamma K.\otimes \Sym^2\Gamma K.)\cong\,\,&
H_k\Tot(K.\otimes K.)=H_k(P.(R/(f))^{\otimes 2})
\\=\,\,& \Lambda^k((f)/(f)^2)\cong
\begin{cases}
R/(f)   &   k=0,1\\
0       &   k>1
\end{cases}
\end{align*}
with the last step given by the isomorphism
\(H_k(P.(R/J)^{\otimes m})\cong\Lambda^k((J/J^2)^{m-1}),\)
see Theorem 5.1 of \cite{Ko1}.

Using Proposition \ref{Kqi} we get
\begin{align*}
H_kN\Sym^3 \Gamma K. \cong H_k(\Kos^3(f))
\end{align*}
now
\begin{align*}
\Kos^3(f)
&=(\Lambda^3(R)\rightarrow
\Lambda^2(R)\otimes R\rightarrow
\Lambda^1(R)\otimes\Sym^2(R)
\rightarrow\Sym^3(R))\\
&=(0\rightarrow 0\rightarrow R\overset{f}\rightarrow R)=K.\text{,}
\end{align*}
hence
\begin{align*}
H_kN\Sym^3 \Gamma K. \cong
\begin{cases}
R/(f)   & k=0 \\
0       & k\ne0\text{.}
\end{cases}
\end{align*}

So the long exact sequence of homologies that we get from the short exact sequence
\[
0\rightarrow NL^3_1\Gamma K. \rightarrow N\Delta(\Gamma K. \otimes \Sym^2 \Gamma K.) \rightarrow N\Sym^3 \Gamma K. \rightarrow 0
\]
is
\[
\xymatrix
{
\ar@{.>}[r]& H_3NL^3_1\Gamma K.\ar[r]       &0\ar[r]                &0\ar `[r]  `[l]  `[llld] `[r] [dll]    &\\
& H_2NL^3_1\Gamma K.\ar[r]  &0\ar[r]                &0\ar `[r]  `[l]  `[llld] `[r] [dll]    &\\
& H_1NL^3_1\Gamma K.\ar[r]  &R/(f)\ar[r]    &0\ar `[r]  `[l]  `[llld] `[r] [dll]    &\\
& H_0NL^3_1\Gamma K.\ar[r]  &R/(f)\ar[r]    &R/(f)\ar[r]        & 0\\
}
\]
and so we get
\[
H_kNL^3_1\Gamma K.\cong
\begin{cases}
0 & k\ne 1\\
R/(f) & k=1\text.
\end{cases}
\]
Similarly we get
\[
H_kNL^3_1\Gamma L.\cong
\begin{cases}
0 & k\ne 1\\
R/(g) & k=1\text.
\end{cases}
\]

Now we work with the short exact sequence of simplicial modules
\[
0\rightarrow \Lambda^3 \Gamma K.\otimes \Lambda^3 \Gamma L. \rightarrow M_{(2,1)}(\Sym^3(\Gamma K.\otimes \Gamma L.))
\rightarrow L^3_1\Gamma K.\otimes L^3_1\Gamma L. \rightarrow 0\text{.}
\]
As above, we rewrite the left and right hand side of this exact sequence
using $\Delta$ and we apply the functor $N$ to get
the following short exact sequence of chain complexes
\begin{align*}
0\rightarrow
N\Delta(\Lambda^3 \Gamma K.\otimes \Lambda^3 \Gamma L.)
&\rightarrow
N M_{(2,1)}(\Sym^3(\Gamma K.\otimes \Gamma L.))\\
&\rightarrow
N\Delta(L^3_1\Gamma K.\otimes L^3_1\Gamma L.)
\rightarrow 0.
\end{align*}
The Eilenberg-Zilber Theorem tells us that
\[H_kN\Delta(\Lambda^3 \Gamma K.\otimes \Lambda^3 \Gamma L.) \cong H_k\Tot(N\Lambda^3 \Gamma K.\otimes N\Lambda^3 \Gamma L.)\]
and
\[H_kN\Delta(L^3_1\Gamma K.\otimes L^3_1\Gamma L.)\cong
H_k\Tot (NL^3_1\Gamma K.\otimes NL^3_1\Gamma L.)\text{.}\]
Now by Propositions \ref{Kqi} and \ref{hTor}
we get
\begin{align*}
H_k\Tot(N\Lambda^3 \Gamma K.\otimes N\Lambda^3 \Gamma L.)
\cong\,\,& H_k\Tot(\tilde\Kos^3(f)\otimes \tilde\Kos^3(g))\text{.}
\end{align*}
Now
\begin{align*}
\tilde\Kos^3(f)
&=(D^3(R)\rightarrow
D^2(R)\otimes R\rightarrow
R\otimes\Lambda^2(R)\rightarrow
\Lambda^3(R))
\\
&=(R\overset{f}\rightarrow R\rightarrow 0\rightarrow 0)=K.[-2]\text{,}
\end{align*}
and similarly \(\tilde\Kos^3(g)=L.[-2]\). So
\begin{align*}
H_k\Tot(N\Lambda^3 \Gamma K.\otimes N\Lambda^3 \Gamma L.)
\cong\,\,& H_k\Tot(K.[-2]\otimes L.[-2])\\
\cong\,\,& H_k(\Tot(K.\otimes L.)[-4])\\
\cong\,\,& H_{k-4}(P.(R/I))\cong
\begin{cases}
0 & k\ne 4\\
R/I & k=4\text{.}
\end{cases}
\end{align*}

Now to calculate \(H_k\Tot (NL^3_1\Gamma K.\otimes NL^3_1\Gamma L.)\)
we note that it is the hypertor functor \(\hTor_k(NL^3_1\Gamma K.,NL^3_1\Gamma L.)\)
and use the hypertor spectral sequence of Proposition \ref{hTor}
taking \(A.=NL^3_1\Gamma K.\) and
\(B.=NL^3_1\Gamma L.\).
But since
\(H_kNL^3_1\Gamma K.\) and
\(H_kNL^3_1\Gamma L.\) are $0$ unless $k=1$ (see above)
this spectral sequence collapses,
with the only (potentially) non-zero terms being when \(q'=q''=1,\) i.e.\ when $q=2$.
These (potentially) non-zero terms are \(\Tor_p(R/(f),R/(g))\).
Taking \(K.\) as a projective resolution of \(R/(f)\) then tensoring throughout by $R/(g)$
we get the chain complex
\begin{align*}
\big{(}0\rightarrow R\otimes R/(g) \overset{f}\rightarrow R\otimes R/(g)\big{)}
=\big{(}0\rightarrow R/(g) \overset{f}\rightarrow R/(g)\big{)}
\end{align*}
which has homology $R/I$ at the $0^{\rm{th}}$ place and $0$ everywhere else.
And so
\begin{align*}
H_k\big{(}\Tot (NL^3_1\Gamma K.\otimes NL^3_1\Gamma L.)\big{)}\cong \Tor_{k-2}(R/(f),R/(g))
\cong
\begin{cases}
0 & k\ne 2\\
R/I & k=2\text{.}
\end{cases}
\end{align*}

So the short exact sequence of chain complexes
\begin{align*}
0\rightarrow
N\Delta(\Lambda^3 \Gamma K.\otimes \Lambda^3 \Gamma L.)
&\rightarrow
N M_{(2,1)}(\Sym^3(\Gamma K.\otimes \Gamma L.))\\
&\rightarrow
N\Delta(L^3_1\Gamma K.\otimes L^3_1\Gamma L.) \rightarrow 0
\end{align*}
gives rise to the following long exact sequence of homologies
\[
\xymatrix
{
\ar@{.>}[r]& 0\ar[r]    &H_5N M_{(2,1)}(\Sym^3(\Gamma K.\otimes \Gamma L.))\ar[r]   &0\ar `[r]  `[l]  `[llld] `[r] [dll]    &\\
& R/I\ar[r]             &H_4N M_{(2,1)}(\Sym^3(\Gamma K.\otimes \Gamma L.))\ar[r]   &0\ar `[r]  `[l]  `[llld] `[r] [dll]    & \\
& 0\ar[r]               &H_3N M_{(2,1)}(\Sym^3(\Gamma K.\otimes \Gamma L.))\ar[r]   &0\ar `[r]  `[l]  `[llld] `[r] [dll]    &\\
& 0\ar[r]               &H_2N M_{(2,1)}(\Sym^3(\Gamma K.\otimes \Gamma L.))\ar[r]   &R/I\ar `[r]    `[l]  `[llld] `[r] [dll]    &\\
& 0\ar[r]               &H_1N M_{(2,1)}(\Sym^3(\Gamma K.\otimes \Gamma L.))\ar[r]   &0\ar `[r]  `[l]  `[llld] `[r] [dll]    &\\
& 0\ar[r]               &H_0N M_{(2,1)}(\Sym^3(\Gamma K.\otimes \Gamma L.))\ar[r]   &0\ar[r]        & 0\\
}
\]
and therefore we get
\begin{align*}
H_kN M_{(2,1)}(\Sym^3(\Gamma K.\otimes \Gamma L.))
\cong
\begin{cases}
R/I & k=2,4\\
0   & \text{otherwise.}
\end{cases}
\end{align*}

Now we work with the short exact sequence of simplicial modules
\begin{align*}
0\rightarrow M_{(2,1)}(\Sym^3(\Gamma K.\otimes \Gamma L.))
&\rightarrow \Sym^3(\Gamma K. \otimes \Gamma L.)\\
&\rightarrow
\Sym^3\Gamma K.\otimes\Sym^3\Gamma L. \rightarrow 0
\end{align*}
the term \(\Sym^3\Gamma K.\otimes\Sym^3\Gamma L.\) is a simplicial module whose $k^{\rm{th}}$
place is \(\Sym^3\Gamma K_k\otimes\Sym^3\Gamma L_k\), but as above
it is more useful to think of it as
the diagonal of the bisimplicial complex whose $(k,l)^{\rm{th}}$ place is
\(\Sym^3\Gamma K_k\otimes\Sym^3\Gamma L_l\).
Applying the functor $N$ we get the following
short exact sequence of chain complexes
\begin{align*}
0\rightarrow NM_{(2,1)}(\Sym^3(\Gamma K.\otimes \Gamma L.))
&\rightarrow N\Sym^3(\Gamma K. \otimes \Gamma L.)\\
&\rightarrow
N\Delta(\Sym^3\Gamma K.\otimes\Sym^3\Gamma L.) \rightarrow 0\text{.}
\end{align*}
Applying the Eilenberg-Zilber Theorem and Propositions \ref{Kqi}
and \ref{hTor} we see
\begin{align*}
H_kN\Delta(\Sym^3\Gamma K.\otimes\Sym^3\Gamma L.)\cong\,\,&
H_k\Tot(N\Sym^3\Gamma K.\otimes N\Sym^3\Gamma L.)\\
\cong\,\,& H_k\Tot(\Kos^3(f)\otimes \Kos^3(g))\text{.}
\end{align*}
Earlier in this proof we showed that \(\Kos^3(f)=K.\) and similarly \(\Kos^3(g)=L.\), so
\begin{align*}
H_k\Tot(\Kos^3(f)\otimes \Kos^3(g))\cong
H_k(\Tot(K.\otimes L.))=H_k(P.(R/I))\text{.}
\end{align*}
Hence
\[H_kN\Delta(\Sym^3\Gamma K.\otimes\Sym^3\Gamma L.)\cong
\begin{cases}
R/I & k=0\\
0 & k\ne 0\text{.}
\end{cases}
\]

And so the short exact sequence of chain complexes
\begin{align*}
0\rightarrow NM_{(2,1)}(\Sym^3(\Gamma K.\otimes \Gamma L.))
&\rightarrow N\Sym^3(\Gamma K. \otimes \Gamma L.)\\
&\rightarrow
N\Delta(\Sym^3(\Gamma K.)\otimes\Sym^3(\Gamma L.)) \rightarrow 0\text{,}
\end{align*}
gives rise to the following long exact sequence
\[
\xymatrix
{
\ar@{.>}[r]& 0\ar[r]    &H_5N\Delta\Sym^3(\Gamma K. \otimes\Gamma L.)\ar[r]     &0\ar `[r]  `[l]  `[llld] `[r] [dll]    &\\
& R/I\ar[r]             &H_4N\Delta\Sym^3(\Gamma K. \otimes\Gamma L.)\ar[r]     &0\ar `[r]  `[l]  `[llld] `[r] [dll]    & \\
& 0\ar[r]               &H_3N\Delta\Sym^3(\Gamma K. \otimes\Gamma L.)\ar[r]     &0\ar `[r]  `[l]  `[llld] `[r] [dll]    &\\
& R/I\ar[r]             &H_2N\Delta\Sym^3(\Gamma K. \otimes\Gamma L.)\ar[r] &0\ar `[r]  `[l]  `[llld] `[r] [dll]    &\\
& 0\ar[r]               &H_1N\Delta\Sym^3(\Gamma K. \otimes\Gamma L.)\ar[r] &0\ar `[r]  `[l]  `[llld] `[r] [dll]    &\\
& 0\ar[r]               &H_0N\Delta\Sym^3(\Gamma K. \otimes\Gamma L.)\ar[r] &R/I\ar[r]      & 0\text{.}\\
}
\]

And hence (as we know \(G_k(R/I)\cong
H_k N\Delta\Sym^3(\Gamma K.\otimes \Gamma L.)\))
we see that
\[G_k(R/I)\cong
\begin{cases}
R/I             & k=0,2,4\\
0               & \text{otherwise,}
\end{cases}
\]
as desired.
\end{proof}

\end{document}